\documentclass[english,preprint,12pt,times]{elsarticle}

\usepackage[T1]{fontenc}
\usepackage{ifthen}
\usepackage{amsmath}
\usepackage{babel}
\usepackage{hyperref}
\usepackage[amsmath,thmmarks,hyperref]{ntheorem}

\let\mathbb=\varmathbb

\provideboolean{usemicrotype}
\setboolean{usemicrotype}{true}

\ifthenelse{\boolean{usemicrotype}}{
        \usepackage[spacing=true,tracking=true,kerning=true,babel]{microtype}
        \microtypecontext{spacing=nonfrench}
}{}

\theoremnumbering{arabic}
\theoremstyle{plain}
\theoremseparator{.}
\RequirePackage{latexsym}
\theoremsymbol{}
\theorembodyfont{\itshape}
\theoremheaderfont{\normalfont\bfseries}
\newtheorem{theorem}{Theorem}[section]
\newtheorem{proposition}[theorem]{Proposition}
\newtheorem{lemma}[theorem]{Lemma}
\newtheorem{corollary}[theorem]{Corollary}

\theoremheaderfont{\upshape}
\newtheorem{claim}{Claim}[theorem]

\theorembodyfont{\upshape}
\theoremsymbol{}
\theoremheaderfont{\normalfont\bfseries}
\newtheorem{definition}[theorem]{Definition}
\newtheorem{remark}[theorem]{Remark}

\theoremstyle{nonumberplain}
\theoremheaderfont{\itshape}
\theorembodyfont{\normalfont}
\theoremsymbol{\ensuremath{\square}}
\RequirePackage{amssymb}
\newtheorem{proof}{Proof}

\theoremsymbol{\ensuremath{\dashv}}
\newtheorem{claimproof}{Proof}

\qedsymbol{\ensuremath{\square}}
\theoremclass{LaTeX}

\newcommand{\id}{\ensuremath{\text{{\rm id}}}}
\newcommand{\Ord}{\ensuremath{\text{{\rm Ord}}}}
\newcommand{\ZFC}{\ensuremath{\text{{\sf ZFC}}}}
\def\forces{\mathrel\|\joinrel\relbar}

\DeclareMathOperator{\rng}{rng}
\DeclareMathOperator{\otp}{otp}
\DeclareMathOperator{\Lim}{Lim}
\newcommand{\LimNoArg}{\ensuremath{\text{{\rm Lim}}}}

\DeclareMathOperator{\cf}{cf}
\DeclareMathOperator{\cof}{cof}

\newcommand{\PFA}{\ensuremath{\text{{\sf PFA}}}}

\newcommand{\Cl}{\ensuremath{\text{{\rm Cl}}}}
\newcommand{\TP}{\ensuremath{\text{{\sf TP}}}}
\newcommand{\SP}{\ensuremath{\text{{\sf SP}}}}
\newcommand{\ITP}{\ensuremath{\text{{\sf ITP}}}}
\newcommand{\ISP}{\ensuremath{\text{{\sf ISP}}}}

\newcommand{\non}{\ensuremath{\lnot}}
\newcommand{\IT}{\ensuremath{\text{{\rm IT}}}}
\newcommand{\IS}{\ensuremath{\text{{\rm IS}}}}

\journal{Annals of Pure and Applied Logic}

\begin{document}

\begin{frontmatter}

\title{The combinatorial essence of supercompactness\tnoteref{t1}}
\tnotetext[t1]{The results of this paper are from the author's doctoral dissertation~\cite{diss} written at the Ludwig Maximilians Universit\"at M\"unchen
under the supervision of Professor Dieter Donder, to whom the author feels greatly indebted.}
\author{Christoph Wei\ss}
\ead{weissch@ma.tum.de}
\address{Mathematisches Institut der Universit\"at M\"unchen, Theresienstr.~39, 80333 M\"unchen, Germany}

\begin{abstract}
We introduce combinatorial principles that characterize strong compactness and supercompactness for inaccessible cardinals but also make sense for successor cardinals.
Their consistency is established from what is supposedly optimal.
Utilizing the failure of a weak version of square, we show that the best currently known lower bounds for the consistency strength of these principles can be applied.
\end{abstract}


\end{frontmatter}

\section{Introduction}
It is a well-known theorem that a cardinal $\kappa$ is weakly compact if and only if it is inaccessible and the $\kappa$-tree property holds, that is, there are no $\kappa$-Aronszajn trees.
By~\cite{mitchell}, the $\omega_2$-tree property can be forced from a weakly compact cardinal and implies $\omega_2$ is weakly compact in $L$.
The tree property thus captures the combinatorial essence of weak compactness, even for successor cardinals.
Similarly, the property that there is no special $\kappa$-Aronszajn tree captures the essence of Mahlo, see~\cite[(1.9)]{todorcevic.partitioning}.

In the present work, we introduce principles $\TP(\kappa, \lambda)$ and $\SP(\kappa, \lambda)$ as well as $\ITP(\kappa, \lambda)$ and $\ISP(\kappa, \lambda)$ that achieve the same for strong compactness and supercompactness respectively.
We present the ideals associated to the principles $\ITP(\kappa, \lambda)$ and $\ISP(\kappa, \lambda)$, prove the consistency of $\ISP(\omega_2, \lambda)$, the strongest of the principles, from a $\lambda^{<\kappa}$-ineffable cardinal, and show $\ITP(\kappa, \lambda)$ implies the failure of a weak form of square, giving lower bounds on its consistency strength.

\subsection*{Notation}

The notation used is mostly standard.
$\Ord$ denotes the class of all ordinals.
For $A \subset \Ord$, $\Lim A$ denotes the class of limit points of $A$.
$\LimNoArg$ stands for $\Lim \Ord$.
If $a$ is a set of ordinals, $\otp a$ denotes the order type of $a$.
For a regular cardinal $\delta$, $\cof \delta$ denotes the class of all ordinals of cofinality $\delta$,
and $\cof(< \delta)$ denotes those of cofinality less than $\delta$.

For forcings, we write $p < q$ to mean $p$ is stronger than $q$.
Names either carry a dot above them or are canonical names for elements of $V$, so that we can confuse sets in the ground model with their names.

The phrases \emph{for large enough $\theta$} and \emph{for sufficiently large $\theta$} will be used for saying that there exists a $\theta'$ such that the sentence's proposition holds for all $\theta \geq \theta'$.

If $\kappa \subset X$, then
\begin{equation*}
	P'_\kappa X \coloneqq \{ x \in P_\kappa X\ |\ \kappa \cap x \in \Ord,\ \langle x, \in \rangle \prec \langle X, \in \rangle \}
\end{equation*}
is club.
For $x \in P'_\kappa X$ we set $\kappa_x \coloneqq \kappa \cap x$.
For $f: P_\omega X \to P_\kappa X$ let $\Cl_f \coloneqq \{ x \in P_\kappa X\ |\ \forall z \in P_\omega x\ f(z) \subset x \}$.
$\Cl_f$ is club, and it is well known that for any club $C \subset P_\kappa X$ there is an $f: P_\omega X \to P_\kappa X$ such that $\Cl_f \subset C$.

If $X \subset X'$, $R \subset P_\kappa X$, $U \subset P_\kappa X'$,
then the projection of $U$ to $X$ is $U \restriction X \coloneqq \{ u \cap X\ |\ u \in U \} \subset P_\kappa X$ and the lift of $R$ to $X'$ is $R^{X'} \coloneqq \{ x' \in P_\kappa X'\ |\ x' \cap X \in R \} \subset P_\kappa X'$.

For sections~\ref{sect.principles},~\ref{sect.ideals}, and~\ref{sect.weak_square}, $\kappa$ and $\lambda$ are assumed to be cardinals, $\kappa \leq \lambda$, and $\kappa$ is regular and uncountable.

\section{Combinatorial principles for strong compactness and supercompactness}\label{sect.principles}
Let us call a sequence $\langle d_a\ |\ a \in P_\kappa \lambda \rangle$ a \emph{$P_\kappa \lambda$-list} if $d_a \subset a$ for all $a \in P_\kappa \lambda$.

\begin{definition}\label{def.P_kappa_lambda.thin}
	Let $D = \langle d_a\ |\ a \in P_\kappa \lambda \rangle$ be a $P_\kappa \lambda$-list.
	\begin{itemize}
		\item $D$ is called \emph{thin} if there is a club $C \subset P_\kappa \lambda$ such that $| \{ d_a \cap c\ |\ c \subset a \in P_\kappa \lambda \} | < \kappa$ for every $c \in C$.
		\item $D$ is called \emph{slender} if for every sufficiently large $\theta$ there is a club $C \subset P_\kappa H_\theta$ such that $d_{M \cap \lambda} \cap b \in M$ for all $M \in C$ and all $b \in M \cap P_{\omega_1} \lambda$.\footnote{Note that this definition is slightly weaker than the one from \cite{diss} as ``for all $b \in M \cap P_\kappa \lambda$'' was replaced by ``for all $b \in M \cap P_{\omega_1} \lambda$.''  However, the proofs in \cite{diss} work for this weaker definition and the resulting stronger principle \ISP\ just the same.}
	\end{itemize}
\end{definition}

\begin{proposition}\label{prop.P_kappa_lambda-thin->slender}
	Let $D$ be a $P_\kappa \lambda$-list.
	If $D$ is thin, then it is slender.
\end{proposition}
\begin{proof}
	Let $C \subset P_\kappa \lambda$ be a club that witnesses $D = \langle d_a\ |\ a \in P_\kappa \lambda \rangle$ is thin.
	Define $g: C \to P_\kappa H_\theta$ by $g(c) \coloneqq	\{ d_a \cap c\ |\ c \subset a \in P_\kappa \lambda \}$.
	Let $\bar{C} \coloneqq \{ M \in C^{H_\theta}\ |\ \forall b \in M \cap P_\kappa \lambda\ \exists c \in M \cap C\ b \subset c,\ \forall c \in M \cap C\ g(c) \subset M \}$.
	Then $\bar{C}$ is club.
	Let $M \in \bar{C}$ and $b \in M \cap P_{\omega_1} \lambda$.
	Then there is $c \in M \cap C$ such that $b \subset c$, so $d_{M \cap \lambda} \cap b = d_{M \cap \lambda} \cap c \cap b \in M$ as $d_{M \cap \lambda} \cap c \in g(c) \subset M$.
	Therefore $\bar{C}$ witnesses $\langle d_a\ |\ a \in P_\kappa \lambda \rangle$ is slender.
\end{proof}

\begin{definition}\label{def.ineffable_branch}
	Let $D = \langle d_a\ |\ a \in P_\kappa \lambda \rangle$ be a $P_\kappa \lambda$-list and $d \subset \lambda$.
	\begin{itemize}
		\item $d$ is called a \emph{cofinal branch of $D$} if for all $a \in P_\kappa \lambda$ there is $z_a \in P_\kappa \lambda$ such that $a \subset z_a$ and $d \cap a = d_{z_a} \cap a$.
		\item $d$ is called an \emph{ineffable branch of $D$} if there is a stationary set $S \subset P_\kappa \lambda$ such that $d \cap a = d_a$ for all $a \in S$.
	\end{itemize}
\end{definition}

Combining these two definitions, we can define the following four combinatorial principles.
\begin{definition}
	\begin{itemize}
		\item $\TP(\kappa, \lambda)$ holds if every thin $P_\kappa \lambda$-list has a cofinal branch.
		\item $\SP(\kappa, \lambda)$ holds if every slender $P_\kappa \lambda$-list has a cofinal branch.
		\item $\ITP(\kappa, \lambda)$ holds if every thin $P_\kappa \lambda$-list has an ineffable branch.
		\item $\ISP(\kappa, \lambda)$ holds if every slender $P_\kappa \lambda$-list has an ineffable branch.
	\end{itemize}
\end{definition}

\begin{remark}\label{remark.inaccessible->thin}
	The reader should note that the principle $\TP(\kappa, \kappa)$ is just the tree property for $\kappa$.
	Also, if $\kappa$ is an inaccessible cardinal, then every $P_\kappa \lambda$-list is thin.
	Therefore $\TP(\kappa, \lambda)$ and $\SP(\kappa, \lambda)$ as well as $\ITP(\kappa, \lambda)$ and $\ISP(\kappa, \lambda)$ are equivalent if $\kappa$ is inaccessible.
	Furthermore this means an inaccessible cardinal $\kappa$ is weakly compact if and only if $\TP(\kappa, \kappa)$ holds, and it is ineffable if and only if $\ITP(\kappa, \kappa)$ holds.
\end{remark}

\begin{remark} The following implications hold.
	\begin{enumerate}
		\item $\ISP(\kappa, \lambda)$ implies $\SP(\kappa, \lambda)$,\label{en.ISP->SP}
		\item $\ISP(\kappa, \lambda)$ implies $\ITP(\kappa, \lambda)$,\label{en.ISP->ITP}
		\item $\ITP(\kappa,\lambda)$ implies $\TP(\kappa,\lambda)$,\label{en.ITP->TP}
		\item $\SP(\kappa,\lambda)$ implies $\TP(\kappa,\lambda)$.\label{en.SP->TP}
	\end{enumerate}
	We will see that~\ref{en.ISP->SP} and~\ref{en.ITP->TP} can not be reversed.
	For if $\kappa$ is a strongly compact cardinal that is not supercompact, then by Theorem~\ref{theorem.TP<->stronglycompact} $\SP(\kappa, \lambda)$ holds for all $\lambda \geq \kappa$, but by Theorem~\ref{theorem.ITP<->supercompact} we have that $\ITP(\kappa, \lambda)$ cannot hold for all $\lambda \geq \kappa$.
	This is also true for smaller $\kappa$.
	One can show that the Mitchell collapse preserves $\SP(\kappa, \lambda)$.
	However, by Theorem~\ref{theorem.ITP_consistency_additional}, if the Mitchell collapse produces a model in which $\ITP(\kappa, \lambda)$ holds, then also in the ground model $\ITP(\kappa, \lambda)$ holds, so that again collapsing a strongly compact cardinal that is not supercompact yields a model in which $\SP(\kappa, \lambda)$ holds but $\ITP(\kappa, \lambda)$ fails.
	Furthermore implication~\ref{en.ISP->ITP} can not be reversed.
	This follows from the fact that the forcing axiom $\PFA(\Gamma_\Sigma)$ from~\cite{koenig.forcing_indestructibility} can be seen to imply $\ITP(\omega_2, \lambda)$ for all $\lambda \geq \omega_2$.
	The paper also shows that $\PFA(\Gamma_\Sigma)$ is consistent with the approachability property holding for $\omega_1$.
	It is easily seen that this contradicts $\ISP(\omega_2, \omega_2)$, so that in any model of $\PFA(\Gamma_\Sigma)$ + ``the approachability property holds for $\omega_1$'' $\ITP(\omega_2, \lambda)$ holds for all $\lambda \geq \omega_2$ but $\ISP(\omega_2, \omega_2)$ fails.
\end{remark}

Jech~\cite{jech.combinatorial_problems} was the first to consider generalizations of the concept of a tree to $P_\kappa \lambda$-lists.
He gave the following characterization of strong compactness.
\begin{theorem}\label{theorem.jech}
	The following are equivalent.
	\begin{enumerate}
		\item $\kappa$ is strongly compact.
		\item For every $\lambda\geq\kappa$, every $P_\kappa \lambda$-list has a branch.
	\end{enumerate}
\end{theorem} 
Shortly after, Magidor~\cite{magidor.characterization_supercompact} extended Jech's result to supercompactness with the next theorem.
\begin{theorem}\label{theorem.magidor}
	The following are equivalent.
	\begin{enumerate}
		\item $\kappa$ is supercompact.
		\item For every $\lambda\geq\kappa$, every $P_\kappa \lambda$-list has an ineffable branch.
	\end{enumerate}
\end{theorem}

By Remark~\ref{remark.inaccessible->thin} we can rephrase Theorems~\ref{theorem.jech} and~\ref{theorem.magidor} in the following way.
\begin{theorem}\label{theorem.TP<->stronglycompact}
	Suppose $\kappa$ is inaccessible.
	Then $\kappa$ is strongly compact if and only if\/ $\TP(\kappa, \lambda)$ holds for every $\lambda \geq \kappa$.
\end{theorem} 
\begin{theorem}\label{theorem.ITP<->supercompact}
	Suppose $\kappa$ is inaccessible.
	Then $\kappa$ is supercompact if and only if\/ $\ITP(\kappa, \lambda)$ holds for every $\lambda \geq \kappa$.
\end{theorem}

The advantage of these new formulations is that $\TP(\kappa, \lambda)$ and $\ITP(\kappa, \lambda)$ are not limited to
inaccessible cardinals, as we will see in section~\ref{sect.consistency}.

\section{The corresponding ideals}\label{sect.ideals}
The principles $\ITP(\kappa, \lambda)$ and $\ISP(\kappa, \lambda)$ have ideals canonically associated to them.

\begin{definition}
	Let $A \subset P_\kappa \lambda$ and
	let $D = \langle d_a\ |\ a \in P_\kappa \lambda \rangle$ be a $P_\kappa \lambda$-list.
	$D$ is called \emph{$A$-effable} if for every $S \subset A$ that is stationary in $P_\kappa \lambda$ there are $a, b \in S$ such that $a \subset b$ and	$d_a \neq d_b \cap a$.
	$D$ is called \emph{effable} if it is $P_\kappa \lambda$-effable.
\end{definition}

\begin{definition}
	We let
	\begin{align*}
		I_\IT[\kappa, \lambda] & \coloneqq \{ A \subset P_\kappa \lambda\ |\ \text{there exists a thin $A$-effable $P_\kappa \lambda$-list} \},\\
		I_\IS[\kappa, \lambda] & \coloneqq \{ A \subset P_\kappa \lambda\ |\ \text{there exists a slender $A$-effable $P_\kappa \lambda$-list} \}.
	\end{align*}
	By $F_\IT[\kappa, \lambda]$ and $F_\IS[\kappa, \lambda]$ we denote the filters associated to $I_\IT[\kappa, \lambda]$
	and $I_\IS[\kappa, \lambda]$ respectively.
\end{definition}
Note that $\ITP(\kappa, \lambda)$ and $\ISP(\kappa, \lambda)$ now say that $I_\IT[\kappa, \lambda]$ and $I_\IS[\kappa, \lambda]$ are proper ideals respectively.
By Proposition~\ref{prop.P_kappa_lambda-thin->slender} we have $I_\IT[\kappa, \lambda] \subset I_\IS[\kappa, \lambda]$.

\begin{proposition}\label{prop.I_IT.normal}
	$I_\IT[\kappa, \lambda]$ and $I_\IS[\kappa, \lambda]$ are normal ideals on $P_\kappa \lambda$.
\end{proposition}
\begin{proof}
	Suppose $D \subset P_\kappa \lambda$ and $g: D \to \lambda$ is regressive.
	Set $A_\gamma \coloneqq {g^{-1}}'' \{ \gamma \}$.
	Let $f: \lambda \times \lambda \to \lambda$ be bijective, and define $f_{\alpha_1}: \lambda \to \lambda$ by $f_{\alpha_1} ( \alpha_0 ) \coloneqq f(\alpha_0, \alpha_1)$.
	We show that if $A_\gamma \in I_\IT[\kappa, \lambda]$ for all $\gamma < \lambda$, then $D \in I_\IT[\kappa, \lambda]$, and that if $A_\gamma \in I_\IS[\kappa, \lambda]$ for all $\gamma < \lambda$, then $D \in I_\IS[\kappa, \lambda]$.

	In the thin case, that is, if $A_\gamma \in I_\IT[\kappa, \lambda]$ for all $\gamma < \lambda$, let $\langle d^\gamma_a\ |\ a \in P_\kappa \lambda \rangle$ be a thin $A_\gamma$-effable $P_\kappa \lambda$-list for $\gamma < \lambda$.
	Let $C^\gamma \subset P_\kappa \lambda$ be a club witnessing $\langle d^\gamma_a\ |\ a \in P_\kappa \lambda \rangle$ is thin.
	Set $C \coloneqq \Delta_{\gamma < \lambda} C^\gamma$.
	We may assume that for all $a \in C$ and all $\alpha_0, \alpha_1 < \lambda$
	\begin{equation}\label{prop.I_IT.normal.eq1}
		f(\alpha_0, \alpha_1) \in a \leftrightarrow \alpha_0, \alpha_1 \in a.
	\end{equation}
	For $a \in C \cap D$ set
	\begin{equation*}
		d_a \coloneqq f_{g(a)}'' d^{g(a)}_a,
	\end{equation*}
	and set $d_a \coloneqq \emptyset$ for $a \in P_\kappa \lambda - ( C \cap D )$.
	If $c \in C$ and $a \in C \cap D$ are such that $c \subset a$ and $g(a) \notin c$, then
	\begin{equation}\label{prop.I_IT.normal.eq2}
		d_a \cap c = \emptyset.
	\end{equation}
	For if $g(a) \notin c$, then by~(\ref{prop.I_IT.normal.eq1}) we have $d_a \cap c = f_{g(a)}'' d^{g(a)}_a \cap c \subset \rng f_{g(a)} \cap c = \emptyset$.
	Thus for fixed $c \in C$ we have $\{ d_a \cap c\ |\ c \subset a \in C \cap D \} \subset \{ f_\gamma'' d^\gamma_a \cap c\ |\ \gamma \in c,\ c \subset a \in C \cap A_\gamma \} \cup \{ \emptyset \}$.
	For $\gamma \in c$ we have $c \in C^\gamma$ and thus, as $C^\gamma$ witnesses $\langle d^\gamma_a\ |\ a \in P_\kappa \lambda \rangle$ is thin, $| \{ d^\gamma_a \cap c\ |\ c \subset a \in C \cap A_\gamma \} | < \kappa$.
	Therefore $| \{ d_a \cap c\ |\ c \subset a \in P_\kappa \lambda\} | < \kappa$, which shows $\langle d_a\ |\ a \in P_\kappa \lambda \rangle$ is thin.

	If $A_\gamma \in I_\IS[\kappa, \lambda]$ for all $\gamma < \lambda$, let $\langle d^\gamma_a\ |\ a \in P_\kappa \lambda \rangle$ be a slender $A_\gamma$-effable $P_\kappa \lambda$-list for $\gamma < \lambda$.
	Let $C^\gamma \subset P'_\kappa H_\theta$ be a club witnessing $\langle d^\gamma_a\ |\ a \in P_\kappa \lambda \rangle$ is slender, where $\theta$ is some large enough cardinal.
	Set $C \coloneqq \Delta_{\gamma < \lambda} C^\gamma$.
	We can again assume that for all $M \in C$ and $\alpha_0, \alpha_1 < \lambda$ $f(\alpha_0, \alpha_1) \in M \leftrightarrow \alpha_0, \alpha_1 \in M$.
	In addition, we may require that
	\begin{equation}\label{prop.I_IT.normal.eq3}
		\langle M, \in, f \restriction (M \times M) \rangle \prec \langle H_\theta, \in, f \rangle
	\end{equation}
	for every $M \in C$.
	As above we define $d_a \coloneqq f_{g(a)}'' d^{g(a)}_a$ for $a \in (C \restriction \lambda) \cap D$ and let $d_a \coloneqq \emptyset$ otherwise.
	By the same argument that led to~(\ref{prop.I_IT.normal.eq2}), we have
	\begin{equation}\label{prop.I_IT.normal.eq4}
		d_a \cap b = \emptyset
	\end{equation}
	if $b \in P_\kappa \lambda$, $a \in (C \restriction \lambda) \cap D$, $b \subset a$, and $g(a) \notin b$.
	To show $\langle d_a\ |\ a \in P_\kappa \lambda \rangle$ is slender, let $M \in C$ and $b \in M \cap P_{\omega_1} \lambda$.
	Set $a \coloneqq M \cap \lambda$.
	If $M \notin D$, then $d_a \cap b \subset d_a = \emptyset \in M$, so assume $M \in D$.
	Then $d_a \cap b = f_{g(a)}'' d^{g(a)}_a \cap b = f_{g(a)}'' ( d^{g(a)}_a \cap {f^{-1}_{g(a)}}'' b )$.
	If $g(a) \notin b$, then by~(\ref{prop.I_IT.normal.eq4}) $d_a \cap b = \emptyset \in M$, so suppose $g(a) \in b$.
	Then ${f^{-1}_{g(a)}}'' b = b$, so by the slenderness of $\langle d^{g(a)}_{\tilde{a}}\ |\ \tilde{a} \in P_\kappa \lambda \rangle$ we have $d^{g(a)}_a \cap {f^{-1}_{g(a)}}'' b \in M$.
	Thus, as $g(a) \in b \subset M$, by~(\ref{prop.I_IT.normal.eq3}) $d_a \cap b = f_{g(a)}'' ( d^{g(a)}_a \cap {f^{-1}_{g(a)}}'' b ) \in M$.

	In both cases we arrived at a $P_\kappa \lambda$-list $\langle d_a\ |\ a \in P_\kappa \lambda \rangle$ such that for a club $C \subset P_\kappa \lambda$ that is closed under $f$ and $f^{-1}$ we have
	\begin{equation*}
		d_a = f_{g(a)}'' d^{g(a)}_a
	\end{equation*}
	for every $a \in C \cap D$, and $d_a = \emptyset$ for $a \in P_\kappa \lambda - (C \cap D)$.
	Suppose that $D \notin I_\IT[\kappa, \lambda]$ for the thin case or $D \notin I_\IS[\kappa, \lambda]$ for the slender case.
	Then there are $S \subset C \cap D$ stationary in $P_\kappa \lambda$ and $d \subset \lambda$ such that $d_a = d \cap a$ for all $a \in S$.
	Since $g$ is regressive we may assume $S \subset A_\gamma$ for some $\gamma < \lambda$.
	But then for $\tilde{d} \coloneqq {f^{-1}_\gamma}'' d$ and $a \in S$ it holds that
	\begin{equation*}
		d^\gamma_a
		= {f^{-1}_\gamma}'' f_\gamma'' d^\gamma_a
		= {f^{-1}_\gamma}'' d_a
		= {f^{-1}_\gamma}'' (d \cap a)
		= {f^{-1}_\gamma}'' d \cap {f^{-1}_\gamma}'' a
		= \tilde{d} \cap a,
	\end{equation*}
	contradicting $\langle d^\gamma_a\ |\ a \in P_\kappa \lambda \rangle$ being effable.
\end{proof}

It is standard to verify that if $\lambda < \lambda'$, then $I_\IT[\kappa, \lambda] \subset \{ A' \restriction \lambda\ |\ A' \in I_\IT[\kappa, \lambda'] \}$ and $I_\IS[\kappa, \lambda] \subset \{ A' \restriction \lambda\ |\ A' \in I_\IS[\kappa, \lambda'] \}$.
This implies the following proposition.
\begin{proposition}\label{prop.ITP_nach_unten}
	Suppose $\lambda \leq \lambda'$.
	Then $\ITP(\kappa, \lambda')$ implies $\ITP(\kappa, \lambda)$, and $\ISP(\kappa, \lambda')$ implies $\ISP(\kappa, \lambda)$.
\end{proposition}

It is easy to check $\cof \omega \cap \kappa \in I_\IT[\kappa, \kappa]$.
The following theorem is the two cardinal analog of this observation.
\begin{theorem}\label{theorem.omega_cofinal_effable}
	Suppose $\cf \lambda \geq \kappa$.
	Then
	\begin{equation*}
		\{ a \in P_\kappa \lambda\ |\ \Lim a \cap \cof \omega \subset a \} \in F_\IT [\kappa, \lambda].
	\end{equation*}
\end{theorem}
\begin{proof}
	Let $A \coloneqq \{ a \in P_\kappa \lambda\ |\ \exists \eta_a \in \Lim a - a\ \cf \eta_a = \omega \}$ and for $a \in A$ let $\eta_a$ be a witness.
	For $\delta \in \cof \omega \cap \lambda$ let $\langle d^\delta_\nu\ |\ \nu < \tau_\delta \rangle$ be an enumeration of $\{ d \subset \delta\ |\ \otp d = \omega,\ \sup d = \delta \}$.
	For $a \in P_\kappa \lambda$ and $\delta \in \Lim a \cap \cof \omega$ let
	\begin{equation*}
		\nu^\delta_a \coloneqq \min \{ \nu < \tau_\delta\ |\ \sup(d^\delta_\nu \cap a) = \delta \}.
	\end{equation*}
	For $a \in A$ set
	\begin{equation*}
		d_a \coloneqq d^{\eta_a}_{\nu^{\eta_a}_a} \cap a,
	\end{equation*}
	and for $a \in P_\kappa \lambda - A$ let $d_a \coloneqq \emptyset$.

	Then $\langle d_a\ |\ a \in P_\kappa \lambda \rangle$ is $A$-effable, for suppose there were a cofinal $U \subset A$ and a $d \subset \lambda$ such that $d_a = d \cap a$ for all $a \in U$.
	Let $a \in U$.
	Since $\cf \lambda \geq \kappa$ there exists $b \in U$ such that $a \cup \Lim a \subset b$.
	But then $\otp (d_b \cap a) < \omega$, contradicting $d_b \cap a = d_a$.

	$\langle d_a\ |\ a \in P_\kappa \lambda \rangle$ is also thin, for let $a \in P_\kappa \lambda$.
	Let
	\begin{equation*}
		B_a \coloneqq \{ d^\delta_{\nu^\delta_a} \cap a\ |\ \delta \in \Lim a \cap \cof \omega \} \cup P_\omega a.
	\end{equation*}
	Then $|B_a| < \kappa$.
	Let $b \in A$ with $a \subset b$, and suppose $d_b \cap a \notin P_\omega a$.
	Since $a \subset b$, we have $\nu^\delta_b \leq \nu^\delta_a$ for all $\delta \in \Lim a \cap \cof \omega$.
	Because $|d_b \cap a| = \omega$ we also have that $d^{\eta_b}_{\nu^{\eta_b}_b} \cap a = d_b \cap a$ is unbounded in $\eta_b$.
	Therefore $\nu^{\eta_b}_a \leq \nu^{\eta_b}_b$, so that $\nu^{\eta_b}_a = \nu^{\eta_b}_b$.
	But this means $d_b \cap a = d^{\eta_b}_{\nu^{\eta_b}_a} \cap a \in B_a$.
\end{proof}

When $\kappa$ is inaccessible, the filter $F_\IT[\kappa, \lambda]$ has some additional simple but helpful properties.
These will be used in section~\ref{sect.consistency}.
\begin{proposition}\label{prop.I_IT_inaccessible}
	Let $\kappa$ be inaccessible.
	Then
	\begin{equation*}
		\{ a \in P'_\kappa \lambda\ |\ \text{$\kappa_a$ inaccessible} \} \in F_\IT [\kappa, \lambda].
	\end{equation*}
\end{proposition}
\begin{proof}
	As $\kappa$ is inaccessible, $\{ a \in P'_\kappa \lambda\ |\ \text{$\kappa_a$ strong limit} \}$ is club.
	So it remains to show $A \coloneqq \{ a \in P'_\kappa \lambda\ |\ \text{$\kappa_a$ singular} \} \in I_\IT[\kappa, \lambda]$.
	Suppose $A \notin I_\IT[\kappa, \lambda]$, and for $a \in A$ let $d_a \subset a$ be such that $\sup d_a = \kappa_a$, $\otp d_a = \cf \kappa_a$.
	Then there exists a stationary $S \subset A$ such that $d_a = d \cap a$ for all $a \in S$.
	We may assume $\kappa_a = \delta$ for some $\delta < \kappa$ and all $a \in S$.
	But if $a, b \in S$ are such that $a \subset b$ and $\kappa_a < \kappa_b$, then $\otp d_b > \delta$, a contradiction.
\end{proof}

\begin{proposition}\label{prop.I_IT_closure}
	Let $\kappa$ be inaccessible.
	Let $g: P_\kappa \lambda \to P_\kappa \lambda$.
	Then
	\begin{equation*}
		\{ a \in P'_\kappa \lambda\ |\ \forall z \in P_{\kappa_a} a\ g(z) \subset a \} \in F_\IT [\kappa, \lambda].
	\end{equation*}
\end{proposition}
\begin{proof}
	Suppose not.
	Then
	\begin{equation*}
		B \coloneqq \{ a \in P'_\kappa \lambda\ |\ \exists z_a \in P_{\kappa_a} a\ g(z_a) \not\subset a \} \notin I_\IT [\kappa, \lambda].
	\end{equation*}
	So let $S \subset B$ be stationary and $z \subset \lambda$ be such that $z_a = z \cap a$ for all $a \in S$.
	For all $a \in S$ we have $\mu_a \coloneqq |z_a| < \kappa_a$, so there are a stationary $S' \subset S$ and $\mu < \kappa$ such that $\mu_a = \mu$ for all $a \in S'$.

	Suppose $|z| > \mu$.
	Then there is $y \subset z$ such that $|y| = \mu^+ < \kappa$.
	But $S'' \coloneqq \{ a \in S'\ |\ y \subset a \}$ is stationary and for every $a \in S''$ we have $z_a = z \cap a \supset y \cap a = y$, which implies $\mu = \mu_a = |z_a| \geq |y| = \mu^+$, a contradiction.

	Since $S'$ is cofinal, there is an $a \in S'$ such that $z \cup g(z) \subset a$.
	But then $z_a = z \cap a = z$ and $g(z_a) = g(z) \subset a$, so that $a \notin B$, contradicting $S' \subset B$.
\end{proof}

\section{The failure of a weak version of square}\label{sect.weak_square}
We define a weak variant of the square principle that is natural for our application.
It is a ``threaded'' version of Schimmerling's two cardinal square principle that is only defined on a subset $E$ of $\lambda$.
\begin{definition}
	A sequence $\langle \mathcal{C}_\alpha\ |\ \alpha \in \LimNoArg \cap E \cap \lambda \rangle$ is called a
	\emph{$\square_E(\kappa, \lambda)$-sequence} if it satisfies the following properties.
	\begin{enumerate}[(i)]
		\item $0 < |\mathcal{C}_\alpha| < \kappa$ for all $\alpha \in \LimNoArg \cap E \cap \lambda$,
		\item $C \subset \alpha$ is club for all $\alpha \in \LimNoArg \cap E \cap \lambda$ and $C \in \mathcal{C}_\alpha$,
		\item $C \cap \beta \in \mathcal{C}_\beta$ for all $\alpha \in \LimNoArg \cap E \cap \lambda$,
			$C \in \mathcal{C}_\alpha$ and $\beta \in \Lim C$,
		\item there is no club $D \subset \lambda$ such that $D \cap \delta \in \mathcal{C}_\delta$
			for all $\delta \in \Lim D \cap E \cap \lambda$.
	\end{enumerate}
	We say that \emph{$\square_E(\kappa, \lambda)$} holds if there exists a $\square_E(\kappa, \lambda)$-sequence.
	\emph{$\square(\kappa, \lambda)$} stands for $\square_\lambda(\kappa, \lambda)$.
\end{definition}
Note that $\square_{\tau, <\kappa}$ implies $\square(\kappa, \tau^+)$
and that $\square(\lambda)$ is $\square(2, \lambda)$.

\begin{theorem}\label{theorem.ITP->non_square}
	Suppose $\cf \lambda \geq \kappa$ and $\square_{\cof(<\kappa)}(\kappa, \lambda)$ holds.
	Then $\non \ITP(\kappa, \lambda)$.
\end{theorem}
\begin{proof}
	Let $A \coloneqq \{ a \in P_\kappa \lambda\ |\ \Lim a \cap \cof \omega \subset a \}$.
	By Theorem~\ref{theorem.omega_cofinal_effable}, $A \in F_\IT[\kappa, \lambda]$.
	So it remains to show $A \in I_\IT [\kappa, \lambda]$.
	We may assume $\sup a \notin a$ for all $a \in A$.
	Let $\langle \mathcal{C}_\gamma\ |\ \gamma \in \LimNoArg \cap \cof(<\kappa) \cap \lambda \rangle$ be a $\square_{\cof(<\kappa)}(\kappa, \lambda)$-sequence.
	For $\gamma \in \LimNoArg \cap \cof(<\kappa) \cap \lambda$ let $C_\gamma \in \mathcal{C}_\gamma$, and set $d_a \coloneqq C_{\sup a} \cap a$ for $a \in A$, otherwise $d_a \coloneqq \emptyset$.
	Then, since $\Lim a \cap \cof \omega \subset a$,
	\begin{equation}\label{theorem.square->nonITP.eq1}
		\sup d_a = \sup a
	\end{equation}
	for every $a \in A$.

	$\langle d_a\ |\ a \in P_\kappa \lambda \rangle$ is thin, for let $a \in P_\kappa \lambda$.
	Set
	\begin{equation*}
		B_a \coloneqq \{ ( C \cap a ) \cup h\ |\ \exists \eta \in \Lim a\ C \in \mathcal{C}_\eta,\ h \in P_\omega a \} \cup P_\omega a.
	\end{equation*}
	Then $|B_a| < \kappa$.
	Let $b \in A$, $a \subset b$, and suppose $d_b \cap a \notin P_\omega a$.
	Let $\eta \coloneqq \max \Lim (d_b \cap a)$.
	Then $\eta \in \Lim C_{\sup b}$, so there is a $C \in \mathcal{C}_\eta$ such that $d_b \cap \eta = C_{\sup b} \cap b \cap \eta = C \cap b$, so $d_b \cap a \cap \eta = C \cap a$.
	Since $|d_b \cap a - \eta| < \omega$, this means $d_b \cap a = ( C \cap a ) \cup ( d_b \cap a - \eta ) \in B_a$.

	$\langle d_a\ |\ a \in P_\kappa \lambda \rangle$ is also $A$-effable.
	For suppose there were a cofinal $U \subset A$ and $d \subset \lambda$ such that $d_a = d \cap a$ for all $a \in U$.
	Then $d$ is unbounded in $\lambda$ by~(\ref{theorem.square->nonITP.eq1}).
	Let $\delta \in \Lim d \cap \cof (< \kappa) \cap \lambda$.
	We will show $d \cap \delta \in \mathcal{C}_\delta$, which contradicts the fact that $\langle \mathcal{C}_\alpha\ |\ \alpha \in \LimNoArg \cap \cof(<\kappa) \cap \lambda \rangle$ is a $\square_{\cof(<\kappa)}(\kappa, \lambda)$-sequence, thus finishing the proof.
	For every $a \in U$ such that $\delta \in \Lim (d \cap a)$ we have $C_{\sup a} \cap a = d_a = d \cap a$, and thus $\delta \in \Lim C_{\sup a}$, so that there is a $C_a \in \mathcal{C}_\delta$ such that $d \cap a \cap \delta = C_a \cap a$.
	But since $|\mathcal{C}_\delta| < \kappa$, there is a cofinal $U' \subset \{ a \in U\ |\ \delta \in \Lim (d \cap a) \}$ such that $C_a = C$ for some $C \in \mathcal{C}_\delta$ and all $a \in U'$.
	But then we have $d \cap \delta \cap a = C \cap a$ for all $a \in U'$, which means $d \cap \delta = C \in \mathcal{C}_\delta$.
\end{proof}

As a corollary, we get a well-known result originally due to Solovay~\cite{solovay}.
\begin{corollary}\label{cor.supercompact->non_square}
	Suppose $\kappa$ is supercompact.
	Then $\non \square_{\cof(<\kappa)}(\kappa, \lambda)$ for all $\kappa \leq \lambda$ with $\cf \lambda \geq \kappa$.
	In particular $\non \square(\lambda)$ for all $\lambda \geq \kappa$ with $\cf \lambda \geq \kappa$.
\end{corollary}
\begin{proof}
	This follows directly from Theorem~\ref{theorem.ITP<->supercompact} and Theorem~\ref{theorem.ITP->non_square}.
\end{proof}

\section{Consistency results}\label{sect.consistency}
\begin{definition} Let $V\subseteq W$ be a pair of transitive models of\/ \ZFC.
	\begin{itemize} 
		\item $(V,W)$ satisfies the $\mu$-covering property if the class $P_\mu^V V$ is cofinal in $P_\mu^W V$, that is, for every $x \in W$ with $x \subset V$ and $|x| < \mu$ there is $z \in P_\mu^V V$ such that $x \subset z$.
		\item $(V,W)$ satisfies the $\mu$-approximation property if for all $x \in W$, $x \subset V$, it holds that if $x \cap z \in V$ for all $z \in P_\mu^V V$, then $x \in V$.
	\end{itemize}
	A forcing $\mathbb{P}$ is said to satisfy the $\mu$-covering property or the $\mu$-approximation property if for every $V$-generic $G \subset \mathbb{P}$ the pair $(V, V[G])$ satisfies the $\mu$-covering property or the $\mu$-approximation property respectively.
\end{definition}

The following theorem was originally discovered by Mitchell~\cite{mitchell}.
We cite~\cite{krueger}, where it is presented in the more modern way we use.
The reader should note we use the convention that conditions are only defined on their support.
\begin{theorem}\label{theorem.existence_iteration}
	Let $\kappa$ be inaccessible, $\tau < \kappa$ be regular and uncountable.
	Then there exists an iteration $\langle \mathbb{P}_\nu\ |\ \nu \leq \kappa \rangle$ such that forcing with $\mathbb{P}_\kappa$ preserves all cardinals less than or equal to $\tau$,
	$\forces_\kappa \kappa = \tau^+$ and for $\eta = 0$ and every inaccessible $\eta \leq \kappa$
	\begin{enumerate}[(i)]
		\item $\mathbb{P}_\eta$ is the direct limit of $\langle \mathbb{P}_\nu\ |\ \nu < \eta \rangle$ and $\eta$-cc,\label{theorem.approx-iteration.en1}
		\item if\/ $\mathbb{P}_\kappa = \mathbb{P}_\eta * \dot{\mathbb{Q}}$, then $\forces_\eta \dot{\mathbb{Q}}\ \text{satisfies the $\omega_1$-approximation property}$,\label{theorem.approx-iteration.en3}
		\item for every $\nu < \eta$, $\mathbb{P}_\nu$ is definable in $H_\eta$ from the parameters $\tau$ and $\nu$,\label{theorem.approx-iteration.en2}
		\item $\mathbb{P}_\eta$ satisfies the $\omega_1$-covering property.\label{theorem.approx-iteration.en4}
	\end{enumerate}
\end{theorem}

The next is a standard lemma which we will need.
\begin{lemma}\label{lemma.iteration_earlier}
	Let $\kappa > \omega$ be regular, $\mathbb{P}_\kappa$ be the direct limit of an iteration $\langle \mathbb{P}_\nu\ |\ \nu < \kappa \rangle$.
	Suppose $\mathbb{P}_\kappa$ is $\kappa$-cc.
	Let $p \in \mathbb{P}_\kappa$ and $\dot{x} \in V^{\mathbb{P}_\kappa}$ such that $p \forces \dot{x} \in P_\kappa V$.
	Then there is $\rho < \kappa$ such that $p \forces \dot{x} \in V[\dot{G}_\rho]$.
\end{lemma}

Recall from \cite{jech.combinatorial_problems} that $\kappa$ is called \emph{$\lambda$-ineffable} if every $P_\kappa \lambda$-list has an ineffable branch.
\begin{theorem}\label{theorem.ISP_consistency}
	Let $\kappa$, $\lambda$ be cardinals,
	$\tau$ regular uncountable, $\tau < \kappa \leq \lambda$, and
	$\langle \mathbb{P}_\nu\ |\ \nu \leq \kappa \rangle$ be an iteration such that for all inaccessible $\eta \leq \kappa$
	\begin{enumerate}[(i)]
		\item $\mathbb{P}_\eta$ is the direct limit of $\langle \mathbb{P}_\nu\ |\ \nu < \eta \rangle$ and $\eta$-cc,\label{theorem.ISP_consistency.en1}
		\item if\/ $\mathbb{P}_\kappa = \mathbb{P}_\eta * \dot{\mathbb{Q}}$, then $\forces_\eta \dot{\mathbb{Q}}\ \text{satisfies the $\omega_1$-approximation property}$,\label{theorem.ISP_consistency.en3}
		\item for every $\nu < \eta$, $\mathbb{P}_\nu$ is definable in $H_\eta$ from the parameters $\tau$ and $\nu$,\label{theorem.ISP_consistency.en2}
		\item $\mathbb{P}_\eta$ satisfies the $\omega_1$-covering property.\label{theorem.ISP_consistency.en4}
	\end{enumerate}
	Suppose $\kappa$ is $\lambda^{<\kappa}$-ineffable.
	Then $\forces_\kappa \text{$\ISP(\kappa, \lambda)$}$.
\end{theorem}
\begin{proof}
	Let $G \subset \mathbb{P}_\kappa$ be $V$-generic and work in $V[G]$.
	Let $\langle d_a\ |\ a \in P_\kappa \lambda \rangle$ be a slender $P_\kappa \lambda$-list, and let $C' \subset P_\kappa H_\theta$ be a club witnessing the slenderness of $\langle d_a\ |\ a \in P_\kappa \lambda \rangle$ for some large enough $\theta$.

	For $x \in P_\kappa \lambda$ by Lemma~\ref{lemma.iteration_earlier} there is $\rho_x < \kappa$ such that $x \in V[G_{\rho_x}]$.
	Thus $C \coloneqq \{ M \in C'\ |\ \forall x \in P_\kappa \lambda \cap M\ \rho_x \in M \}$ is such that $P_\kappa \lambda \cap M \subset V[G_{\kappa_M}]$ for all $M \in C$.

	Let $\sigma \coloneqq (\lambda^{<\kappa})^V$.
	Let $\bar{M} \in V$ be such that $\bar{M} \prec H_\theta^V$, $\lambda \cup P^V_\kappa \lambda \subset \bar{M}$, $|\bar{M}|^V = \sigma$.
	Let $C_0 \coloneqq C \restriction \bar{M}$.
	Since $\mathbb{P}_\kappa$ is $\kappa$-cc, there is a $C_1 \in V$ such that $C_1 \subset C_0$ and $V \models C_1 \subset P_\kappa \bar{M}\ \text{club}$.

	Let
	\begin{equation*}
		E \coloneqq \{ M \in C_1\ |\ \tau < \kappa_M,\ \kappa_M\ \text{inaccessible in $V$},\ 
		P^V_\tau (M \cap \lambda) \subset M\}.
	\end{equation*}

	\begin{claim}\label{consistency_claim1}
		If $M \in E$, then $d_{M \cap \lambda} \in V[G_{\kappa_M}]$.
	\end{claim}
	\begin{claimproof}
		Let $z \in P^{V[G_{\kappa_M}]}_{\omega_1} (M \cap \lambda)$.
		$\mathbb{P}_{\kappa_M}$ satisfies the $\omega_1$-covering property by~(\ref{theorem.ISP_consistency.en4}), so there is $b \in P^V_{\omega_1} (M \cap \lambda)$ such that $z \subset b$.
		Let $M' \in C$ be such that $M = M' \cap \bar{M}$.
		Then $b \in M \subset M'$.
		Therefore, by the slenderness of $\langle d_a\ |\ a \in P_\kappa \lambda \rangle$, $d_{M \cap \lambda} \cap b  = d_{M' \cap \lambda} \cap b \in P_\kappa \lambda \cap M' \subset V[G_{\kappa_{M'}}] = V[G_{\kappa_M}]$ and thus $d_{M \cap \lambda} \cap z = d_{M \cap \lambda} \cap b \cap z \in V[G_{\kappa_M}]$.
	
		Let $\mathbb{P}_\kappa = \mathbb{P}_{\kappa_M} * \dot{\mathbb{Q}}$.
		Then $\dot{\mathbb{Q}}^{G_{\kappa_M}}$ satisfies the $\omega_1$-approximation property by~(\ref{theorem.ISP_consistency.en3}), so since $z$ was arbitrary we get $d_{M \cap \lambda} \in V[G_{\kappa_M}]$.
	\end{claimproof}
	
	For $M \in E$ we have $\mathbb{P}_{\kappa_M} \subset M$ by~(\ref{theorem.ISP_consistency.en1}) and~(\ref{theorem.ISP_consistency.en2}).
	By Claim~\ref{consistency_claim1} there is $\dot{d}_M \in V^{\mathbb{P}_{\kappa_M}}$ such that $\dot{d}_M^{G_{\kappa_M}} = d_{M \cap \lambda}$.
	Let
	\begin{equation*}
		D_M \coloneqq \{ \langle p, \alpha, n \rangle\ |\ p \in \mathbb{P}_{\kappa_M},\ \alpha \in M \cap \lambda,\ 
		(n = 0 \land p \forces_{\kappa_M} \alpha \notin \dot{d}_M) \lor (n = 1 \land p \forces_{\kappa_M} \alpha \in \dot{d}_M) \}.
	\end{equation*}
	Then $\langle D_M\ |\ M \in E \rangle \in V$ and $D_M \subset M$.

	Work in $V$.
	Let $f: \bar{M} \to \sigma$ be a bijection.
	If $\lambda > \kappa$, additionally choose $f$ such that $f \restriction \kappa = \id \restriction \kappa$.
	If $\kappa = \lambda$, then $\{ M \in C_1\ |\ f''M = \kappa_M \}$ is club, and we may assume it is $C_1$.
	By Propositions~\ref{prop.I_IT_inaccessible} and~\ref{prop.I_IT_closure}
	\begin{equation*}
		F \coloneqq \{ m \in P'_\kappa \sigma\ |\ \kappa_m\ \text{inaccessible},\ 
		P_\tau (m \cap f'' \lambda) \subset m \} \in F_\IT[\kappa, \sigma].
	\end{equation*}
	As $\kappa$ is $\sigma$-ineffable, there exist a stationary $S' \subset F$ and $d' \subset \sigma$ such that $f''D_{{f^{-1}}''m} = d' \cap m$ for all $m \in S'$ such that ${f^{-1}}''m \in E$.
	But $E = \{ {f^{-1}}''m\ |\ m \in F \} \cap C_1$ by our choice of $f$ or the additional assumption on $C_1$, so for $S \coloneqq \{ {f^{-1}}''m\ |\ m \in S' \cap F \} \cap C_1$ and for $D \coloneqq {f^{-1}}''d'$ we have $D_M = D \cap M$ for all $M \in S$.

	Back in $V[G]$, let $T \coloneqq S \restriction \lambda$ and
	\begin{equation*}
		d \coloneqq \{ \alpha < \lambda\ |\ \exists p \in G\ \langle p, \alpha, 1 \rangle \in D \}.
	\end{equation*}
	
	\begin{claim}\label{theorem.ISP_consistency.claim3}
		If $a \in T$, then $d_a = d \cap a$.
	\end{claim}
	\begin{claimproof}
		If $a \in T$, then $a = M \cap \lambda$ for some $M \in S$.
		But then for $\alpha \in a$, if $\alpha \in d_a = d_{M \cap \lambda} = \dot{d}_M^{G_{\kappa_M}}$, then there is $p \in G_{\kappa_M}$ such that $p \forces_{\kappa_M} \alpha \in \dot{d}_M$.
		Thus $\langle p, \alpha, 1 \rangle \in D_M = D \cap M$, so that $\alpha \in d$ by the definition of $d$.

		By the same argument, if $\alpha \notin d_a$, then $\alpha \notin d$.
	\end{claimproof}
	$T$ is stationary in $V$, so it is also stationary in $V[G]$ since $\mathbb{P}_\kappa$ is $\kappa$-cc.
	Therefore, by Claim~\ref{theorem.ISP_consistency.claim3}, $\langle d_a\ |\ a \in P_\kappa \lambda \rangle$ is not effable.
\end{proof}
Note that if $\kappa$ is $\lambda$-ineffable and $\cf \lambda \geq \kappa$, then by \cite{johnson} it follows that $\lambda^{<\kappa} = \lambda$.
So in this case, Theorem~\ref{theorem.ISP_consistency} shows $\ISP(\kappa, \lambda)$ is forced from the more natural condition that $\kappa$ is $\lambda$-ineffable.

\begin{corollary}\label{corollary.consistency_subtle_ineffable}
	If the theory \ZFC\ + ``there is an ineffable cardinal'' is consistent, then the theory \ZFC\ + $\ISP(\omega_2, \omega_2)$ is consistent.
\end{corollary}
\begin{proof}
	Taking $\tau = \omega_1$, this follows immediately from Theorem~\ref{theorem.existence_iteration}, Theorem~\ref{theorem.ISP_consistency}, and Remark~\ref{remark.inaccessible->thin}.
\end{proof}

\begin{corollary}\label{corollary.consistency_supercompact}
	If the theory \ZFC\ + ``there exists a supercompact cardinal'' is consistent, then the theory \ZFC\ + ``$\ISP(\omega_2, \lambda)$ holds for every $\lambda \geq \omega_2$'' is consistent.
\end{corollary}
\begin{proof}
	This follows from Theorems~\ref{theorem.existence_iteration},~\ref{theorem.ISP_consistency}, and~\ref{theorem.ITP<->supercompact}.
\end{proof}
In Corollaries~\ref{corollary.consistency_subtle_ineffable} and~\ref{corollary.consistency_supercompact}, $\omega_2$ only serves as the minimal cardinal for which the theorems hold true.
One can of course take successors of larger regular cardinals instead.

It is worth noting that, when using the Mitchell forcing from Theorem~\ref{theorem.existence_iteration}, Corollary~\ref{corollary.consistency_supercompact} and, when $\cf \lambda \geq \kappa$, Theorem~\ref{theorem.ISP_consistency} were best possible, as shows the next theorem.
Its proof can be found in \cite[Theorem 2.3.5]{diss} or \cite{joint}, where similar ``pull back'' theorems are used in a more general setting.
\begin{theorem}\label{theorem.ITP_consistency_additional}
	Let $V \subset W$ be a pair of models of\/ \ZFC\ that satisfies the $\kappa$-covering property and the $\tau$-approximation property for some $\tau < \kappa$, and suppose $\kappa$ is inaccessible in $V$.
	Then
	\begin{equation*}
		P^W_\kappa \lambda - P^V_\kappa \lambda \in I^W_\IT[\kappa, \lambda],
	\end{equation*}
	which furthermore implies
	\begin{equation*}
		F_\IT^V[\kappa, \lambda] \subset F^W_\IT[\kappa, \lambda].
	\end{equation*}
	So in particular, if $W \models \ITP(\kappa, \lambda)$, then $V \models \ITP(\kappa, \lambda)$.
\end{theorem}

We proceed to give lower bounds on the consistency strength of our combinatorial principles.
We first consider the one cardinal variant, showing Corollary~\ref{corollary.consistency_subtle_ineffable} was best possible.

The next lemma is usually only given in its weaker version where $\kappa$ is required to be weakly compact.
\begin{lemma}\label{lemma.initial_segments_in_L}
	Suppose $\kappa$ is regular uncountable and the tree property holds for $\kappa$.
	Let $A \subset \kappa$.
	If $A \cap \alpha \in L$ for all $\alpha < \kappa$, then $A \in L$.
\end{lemma}
\begin{proof}
	Let $\delta \coloneqq \kappa + \omega$.
	By~\cite[Proposition~5.3]{mitchell}, $\kappa$ is inaccessible in $L[A]$.
	By the usual argument, one proves there exists a nonprincipal $\kappa$-complete ultrafilter $U$ on $P^{L[A]} \kappa \cap L_\delta[A]$, see~\cite[Proof of Theorem~5.9]{mitchell}.
	Let $M$ be the transitive collapse of the internal ultrapower of $L_\delta[A]$ by $U$, and let $j: L_\delta[A] \to M$ be the corresponding embedding.
	Then $j$ has critical point $\kappa$.
	As $L_\delta[A] \models V = L[A]$, we have $M \models V = L[j(A)]$, so $M = L_\gamma[j(A)]$ for some limit ordinal $\gamma \geq \delta$.
	It holds that $L_\delta[A] \models \forall \alpha < \kappa\ A \cap \alpha \in L$, so $L_\gamma[j(A)] \models \forall \alpha < j(\kappa)\ j(A) \cap \alpha \in L$, so in particular $L_\gamma[j(A)] \models A = j(A) \cap \kappa \in L$.
	Therefore really $A \in L$.
\end{proof}

\begin{theorem}\label{theorem.ITP->ineffable_in_L}
	Suppose $\kappa$ is regular and uncountable.
	If\/ $\ITP(\kappa, \kappa)$ holds, then $L \models \kappa$ is ineffable.
\end{theorem}
\begin{proof}
	Again by~\cite[Proposition~5.3]{mitchell}, $\kappa$ is inaccessible in $L$.

	Let $\langle d_\alpha\ |\ \alpha < \kappa \rangle \in L$.
	Then $\{ d_\alpha \cap \beta\ |\ \alpha \leq \kappa \} \subset P^L\beta$.
	So $\langle d_a\ |\ a \in P_\kappa \lambda \rangle$, where $d_a = \emptyset$ if $a \notin \kappa$, is thin as $|P^L\beta| < \kappa$.
	Thus by $\ITP(\kappa, \kappa)$ there is a $d \subset \kappa$ such that $d_\alpha = d \cap \alpha$ for stationarily many $\alpha < \kappa$.
	This also means $d \cap \gamma \in L$ for all $\gamma < \kappa$.
	Therefore $d \in L$ by Lemma~\ref{lemma.initial_segments_in_L}.
	Since $\{\alpha < \kappa\ |\ d_\alpha = d \cap \alpha \} \in L$ is also stationary in $L$, the proof is finished.
\end{proof}

The best known lower bounds for the consistency strength of $\ITP(\kappa, \lambda)$ are derived from the failure of square.
The following theorem is due to Jensen, Schimmerling, Schindler, and Steel~\cite{jensen.schimmerling.schindler.steel}.
\begin{theorem}\label{theorem.strong_and_woodin}
	Suppose $\lambda \geq \omega_3$ is regular such that $\eta^\omega < \lambda$ for all $\eta < \lambda$.
	If $\non \square(\lambda)$ and $\non \square_\lambda$, then there exists a sharp for a proper class model with a proper class of strong cardinals and a proper class of Woodin cardinals.
\end{theorem}

\begin{corollary}\label{corollary.consistency-implication}
	The consistency of\/ \ZFC\ + ``there is a $\kappa^+$-ineffable cardinal $\kappa$'' implies the consistency of\/ \ZFC\ + ``there is a proper class of strong cardinals and a proper class of Woodin cardinals.''
\end{corollary}
\begin{proof}
	If $\kappa$ is $\kappa^+$-ineffable, then it is inaccessible and thus $\eta^\omega < \kappa$ for all $\eta < \kappa$.
	By Proposition~\ref{prop.ITP_nach_unten}, $\ITP(\kappa, \kappa)$ holds.
	By Theorem~\ref{theorem.ITP->non_square}, $\ITP(\kappa, \kappa)$ and $\ITP(\kappa, \kappa^+)$ imply $\non \square(\kappa)$ and $\non \square(\kappa^+)$, so by Theorem~\ref{theorem.strong_and_woodin} there is an inner model with a proper class of strong cardinals and a proper class of Woodin cardinals.
\end{proof}

\begin{corollary}
	Suppose $\kappa$ is regular uncountable and $\lambda \geq \omega_3$ is such that $\cf \lambda \geq \kappa$ and $\eta^\omega < \lambda$ for all $\eta < \lambda$.
	If\/ $\ITP(\kappa, \lambda^+)$ holds, then there exists an inner model with a proper class of strong cardinals and a proper class of Woodin cardinals.
\end{corollary}
\begin{proof}
	This follows from Proposition~\ref{prop.ITP_nach_unten}, Theorem~\ref{theorem.ITP->non_square}, and Theorem~\ref{theorem.strong_and_woodin}.
\end{proof}

\section{Conclusion}
The reader will have noted that one could also define principles corresponding to $\lambda$-almost ineffability.
However, by~\cite{carr} $\lambda$-ineffability and $\lambda$-almost ineffability both characterize supercompactness, so that considering these principles does not seem to give any new insights.

The main motivation behind the principles we considered is of course the quest for an inner model for a supercompact cardinal.
So far the most interesting applications of the principles can be found in~\cite{joint}, which shows the following.
Suppose $\kappa$ is an inaccessible cardinal and $\mathbb{P}$ is an iteration of forcings of size less than $\kappa$ that takes direct limits stationarily often.
If $\mathbb{P}$ forces \PFA\ and $\kappa = \omega_2$, then $\kappa$ is strongly compact.
If $\mathbb{P}$ is additionally required to be proper, then $\kappa$ is necessarily supercompact.
As this is the only known means of constructing models of \PFA\ from large cardinal assumptions, it gives strong heuristic evidence on the lower bound of the consistency strength of \PFA.

\ifthenelse{\boolean{usemicrotype}}{\microtypesetup{spacing=false}}{}
\bibliographystyle{elsarticle-num}
\bibliography{article}

\end{document}